\documentclass[a4paper,10pt]{amsart}
\usepackage[utf8]{inputenc}
\usepackage{mhequ}
\usepackage{bbm}
\usepackage{amsmath,amsthm,amsfonts,amssymb,mathtools,mathrsfs}
\usepackage{comment}
\usepackage{bm}
\usepackage{appendix}
\usepackage{setspace}
\usepackage{thmtools} 
\usepackage[foot]{amsaddr}
\usepackage[pdfdisplaydoctitle,colorlinks,breaklinks,urlcolor=blue,linkcolor=blue,citecolor=blue]{hyperref}
\usepackage[nameinlink,capitalise]{cleveref}

\renewcommand{\aa}[1]{{#1}}

\newcommand{\N}{\mathbb{N}}

\newcommand{\PP}{\mathbb{P}}

\newcommand{\R}{\mathbb{R}}

\newcommand{\T}{\mathbb{T}}
\newcommand{\TN}{\aa{\mathbb{T}^{N}}}

\newcommand{\XX}{\mathcal{X}}

\newcommand{\Z}{\mathbb{Z}}

\newcommand\isf{\Psi^m}
\newcommand\sem{s}

\def\epsilon{\varepsilon}

\renewcommand{\setminus}{\smallsetminus}

\newcommand{\floor}[1]{\left\lfloor #1\right\rfloor}
\newcommand{\set}[1]{\left\{#1\right\}}
\newcommand{\pa}[1]{\left(#1\right)}

\newcommand{\abs}[1]{\left|#1\right|}

\newcommand{\prob}[1]{\mathbb{P}\left(#1\right)}

\newtheorem{theorem}{Theorem}[section]
\newtheorem{definition}[theorem]{Definition}

\newtheorem{lemma}[theorem]{Lemma}
\newtheorem{proposition}[theorem]{Proposition}

\theoremstyle{remark}

\numberwithin{equation}{section}

\title[Random Splitting of Point Vortex Flows]{Random Splitting of Point Vortex Flows}
\author[A. Agazzi]{Andrea Agazzi}
\address{Università di Pisa, Dipartimento di Matematica, 5 Largo Bruno Pontecorvo, 56127 Pisa, Italia.}
\email{\href{mailto:andrea.agazzi at unipi.it}{andrea.agazzi at unipi.it}}
\author[F. Grotto]{Francesco Grotto}
\address{Università di Pisa, Dipartimento di Matematica, 5 Largo Bruno Pontecorvo, 56127 Pisa, Italia.}
\email{\href{mailto:francesco.grotto at unipi.it}{francesco.grotto at unipi.it}}
\author[J. Mattingly]{Jonathan C. Mattingly}
\address{Department of Mathematics and Department of Statistical Science, Duke University, Durham, NC, 27708 USA}
\email{\href{mailto:jonm@duke.edu}{jonm at duke.com}}
\date\today

\begin{document}

\begin{abstract}
    We consider a stochastic version of the point vortex system, in which the fluid velocity advects single vortices intermittently for small random times.
    Such system converges to the deterministic point vortex dynamics as the rate at which single components of the vector field are randomly switched diverges,
    and therefore it provides an alternative discretization of {2D} Euler equations.
    The random vortex system we introduce preserves microcanonical statistical ensembles of the point vortex system, hence constituting a simpler alternative to the latter in the statistical mechanics approach to 2D turbulence.
\end{abstract}

\maketitle

\section{Introduction}\label{sec:introduction}

The Point Vortex (PV) system is a finite-dimensional system of singular ODEs describing the evolution of an incompressible, 2-dimensional fluid in the idealized case where the vorticity, \emph{i.e.} the curl of the velocity field, is concentrated in a finite set of points.

Introduced by Helmholtz in 1858 \cite{Helmholtz1858},
the PV system is known to be well-posed for almost every initial configuration \cite{Durr1982}.
It has been shown to be the limit of solutions of 2D Euler equations \cite{Marchioro1983,Marchioro1988,Marchioro1993} in the well-posedness class $L^\infty$ of the latter, and the PV system itself converges to solutions of 2D Euler equations in a Mean Field scaling regime for initial data in $L^\infty$ \cite{Rosenzweig2022} (\emph{cf.} also \cite{Jabin2018}).

The properties of PV dynamics as a Hamiltonian system with singular interactions have also been the object of extensive research, because of the coexistence of stable and unstable configurations and the presence of singular solutions possibly related to dissipation properties of 2D Euler equations \cite{Grotto2022a}
(\emph{cf.} \cite{Modin2021} for an overview on PV as Hamiltonian dynamics).

Just as in the case of the closely related {2D} Euler PDE dynamics, the main problem concerning PV systems is the long-time asymptotic properties of the solutions. Equilibrium states of PV systems play a prominent role in the statistical mechanics approach to 2D turbulence rooted in the works of Onsager \cite{Onsager1949}, with the convergence towards states exhibiting the formation of coherent structures being the crucial mathematical open problem \cite{Tabeling2002}.

The present note is devoted to a stochastic modification of the PV system inspired by the \emph{random splitting} technique recently developed in \cite{Agazzi2022,Agazzi2023}. We will prove that the stochastic vortex flow we exhibit is in fact a regularized version of the deterministic PV dynamics converging to the latter in the limit of small regularization parameter. This in turn implies convergence towards solutions of 2D Euler equations in sight of the aforementioned results. Unlike the original PV system, the stochastic dynamics we propose is well-defined for \emph{all} initial configurations, the convergence to the deterministic system holding up to the time of eventual singularities, as in other versions of PV dynamics regularized by the introduction of noise \cite{Flandoli2011,Grotto2023}.

The most important feature of the stochastic dynamics we propose is that it preserves the same kinetic energy functional (\emph{i.e.} the PV Hamiltonian) as the original deterministic flow. To the best of our knowledge, this is the first desingularization method for PV dynamics that preserves such a crucial first integral of motion, possibly opening the way for a new approach to the study of microcanonical ensembles of PVs and their relation with 2D turbulence. We defer a proper discussion to \Cref{sec:statmech}, after having established a rigorous construction and our main results in \cref{sec:main}.

\section{Splitting Vortex Flows}\label{sec:main}

We consider the dynamics on the periodic space domain $\T\simeq [0,1)^2$ and establish our results on the finite time interval $t\in [0,1]$. All the forthcoming arguments can be easily adapted to the general case of PV dynamics on smooth surfaces with or without boundaries. Throughout, we define $k^\perp := (-k_2, k_1)$ for $k = (k_1,k_2) \in \mathbb R^2$, extending this notation naturally to the differential operator $\nabla := (\partial_1, \partial_2)$, and denote by $|\cdot|$, respectively $\|\cdot\|$ the Euclidean and induced operator norm.

\subsection{Deterministic Vortex Dynamics}
A system of $N$ point vortices with \emph{intensities} $\xi_1,\dots,\xi_N\in \R \setminus \{0\}$ and distinct positions
\begin{equ}
    x=(x_1,\dots,x_N)\in \XX := \set{x\in \TN: x_i\neq x_j \,\forall i\neq j}\,,
\end{equ}
evolves according to the dynamics
\begin{equ}\label{eq:pv}
    \dot x_i = v_i(x), \quad v_i(x)=\sum_{j\neq i} \xi_j K(x_i-x_j)\,,
\end{equ}
where
\begin{equ}
    K:\T\setminus (0,0) \to \R^2,\quad K(x)=\nabla^\perp \Delta^{-1}(x)
    =\frac1{2\pi} \sum_{k\in \Z^2\setminus (0,0)} \frac{k^\perp}{|k|^2} e^{2\pi k\cdot x}\,,
\end{equ}
is the 2D Biot-Savart kernel, whose action on a vorticity distribution returns the corresponding velocity field of the fluid, which in turn advects vortex positions.

As proved in \cite{Durr1982}, for almost all initial configurations $x=x(0)\in\XX$ with respect to the product Lebesgue measure on $\TN\supset \XX$ the ODE system \cref{eq:pv} admits a unique solution which is smooth and global in time. With a slight abuse of notation we will denote by $\Phi_t:\TN\to \TN$ the (almost-everywhere defined) solution flow of \cref{eq:pv}, \emph{i.e.} the flow of $ v=(v_1,\dots,v_N)$ on $\XX$.

We will also denote by $\Phi^{(i)}_t:\TN\to \TN$ the flow of a single component of the velocity field, $(0,\dots,v_i,\dots,0)$. For $i> N$, abusing again notation, we will write $\Phi^{(i)}$ implying that the apex is to be considered modulo $N$. Notice that, unlike $\Phi$, each flow $\Phi^{(i)}$ is well-defined for all times at any point $x\in \XX$, because of the particular form of the interaction kernel $K$, which prevents the $i$-th vortex from colliding with any other one.

\subsection{Stochastic Splitting}
Denoting throughout $\floor{y} := \max\{k \in \mathbb N~:~k\leq y\}$,
we define the (stochastically) split PV flow(s) as follows:

\begin{definition}
    Let $m \in \N$ and consider a vector of i.i.d. non-negative random variables $\tau = (\tau_i)_{i = 1}^\infty$ with common distribution $\rho$ having at most exponential tails and satisfying $\mathbb E(\tau_i) = 1$.
    For $t>0$,  define the \emph{jumping stochastic flow}
    \begin{equ}
        \Phi_{t}^m(x) = \Phi_{\tau_{\ell N}/m}^{(N)} \circ \dots \circ \Phi_{\tau_{1}/m}^{(1)}(x),\qquad \ell = \floor{mt}\,,
    \end{equ}
    and the \emph{interpolated stochastic flow} as the solution of
    \begin{equ}
    \frac{d}{dt} \isf_t(x)=N \tau_i v_j(x), \quad i=\floor{Nmt},\quad j=i\,(\mathrm{mod}\, N)\,.
    \end{equ}
    In particular, when $m t$ is an integer,
    \begin{equ}\label{e:psiandphi}
        \isf_t(x)= \Phi_{t}^m(x)\,.
    \end{equ}
\end{definition}

Concretely, we let the single components $\Phi^{(i)}$ of the PV flow act one by one, over small (random) time intervals. The difference between $\Phi^m$ and $\Psi^m$ consists in the former being piecewise constant in time and the latter having continuous trajectories.
Notice that the stochastic flows are well-defined for all initial configurations, even the ones leading to collapse at finite time the dynamics \cref{eq:pv}, since this is the case for every single $\Phi^{(i)}$.

\subsection{Convergence for regularized interaction kernels}\label{ssec:smooth}

Consider the following smooth approximation of the PV interaction kernel:
\begin{equ}
    K_\delta(x)=(1-\chi_\delta(x)) K(x), \quad \delta>0\,,
\end{equ}
with $\chi_\delta\in C^\infty(\T)$ supported by a $\delta$-neighborhood of $0\in\T$ and with $\chi_\delta(0)=1$.
In the present paragraph, we assume that $K_\delta$ replaces $K$ in the definitions of flows $\Phi,\Phi^{(i)},\Phi^m,\isf$,
omitting dependence on $\delta$ for a lighter notation.
Notice that if a solution of the PV dynamics is such that vortices are $\delta$-separated at all times,
that is also a solution of \cref{eq:pv} with $K_\delta$ replacing $K$.

\begin{proposition} \label{p:convKsmooth}
    Let $K_\delta$ replace $K$ in \cref{eq:pv} and the subsequent definitions, and let $x\in \mathcal X$ be fixed. Then
    $\PP$-almost surely, for all $t\in [0,1]$,
    \begin{equ}
       \isf_t(x) \to \Phi_t(x)\qquad \text{as } m\to\infty\,.
    \end{equ}
\end{proposition}
\noindent
Defining
   \begin{equs}\label{e:dT}
   d_\T(x,y)^2 := \sum_{j=1}^N \sum_{i = 1}^2\min\big(\abs{(x_j)_i-(y_j)_i},1-\abs{(x_j)_i-(y_j)_i}\big)^2\,,
   \end{equs}
we first prove convergence of fixed-time marginals of the interpoleted stochastic flow:

\begin{lemma} \label{l:marginals}
    Let the assumptions of \cref{p:convKsmooth} hold and fix $t \in [0,1]$, then for all $\epsilon > 0$
    \begin{equ}
       \prob{\limsup_{m \to \infty}\sup_{x \in \mathcal X}d_\T\pa{\Psi_t^m(x), \Phi_t(x)} > \epsilon } = 0\,.
    \end{equ}
\end{lemma}

\begin{proof}[Proof of \cref{l:marginals}]
    For $\ell\in \mathbb N$, introduce the $\ell$-step jumping flow with timestep $h > 0$:
        \begin{equ}
        \widetilde \Phi_{h}^{\ell}(x) := \widetilde\Phi_h^{(\ell N,1)}(x)\,,\qquad
        \end{equ}
        where for any $1\leq j \leq i $,
        \begin{equ}\widetilde \Phi_h^{(i,j)}(x) :=  \Phi_{h\tau_{i}}^{(i)} \circ \dots \circ \Phi_{h\tau_{j}}^{(j)}(x)\,.
    \end{equ}
    We couple $\widetilde \Phi_{h}^{\ell}(x)$ with $\Phi_t^m(x)$ and $\Psi_t^m(x)$ by setting $h = 1/m$ and identifying the underlying $\tau = (\tau_i)_{i=1}^\infty$, so that whenever $mt \in \mathbb N$, we have
    $
    \widetilde\Phi_{1/m}^{mt}(x) = \Phi_t^m(x) = \Psi_t^m(x)
    $.\\
    We proceed to prove that  for any $t \in [0,1]$ 
    and for any $\epsilon>0$ sufficiently small, we have
    \begin{equs}
      \prob{\limsup_{m \to \infty}\sup_{x \in \mathcal X}d_\T\pa{\Psi_t^m(x), \Phi_t^m(x)} > \frac \epsilon 3} &= 0\,,\label{e:bound0}\\
    \prob{\limsup_{m \to \infty} \sup_{x \in \mathcal X}d_\T\pa{\Phi_t^m(x), \widetilde \Phi_{t/\floor{mt}}^{\floor{mt}}(x)} > \frac \epsilon 3} &= 0\,,\label{e:bound1} \\
       \prob{\limsup_{m \to \infty} \sup_{x \in \mathcal X}d_\T\pa{\widetilde\Phi_{t/\floor{mt}}^{\floor{mt}}(x) , \Phi_t(x)} > \frac \epsilon 3} &= 0\,,\label{e:bound2}
    \end{equs}
    which, combined, yield the desired result.

    Starting from \cref{e:bound2}, we recall from \cite{Agazzi2022} the definition of the flow maps
    \begin{equs}
        S_t f(x)  := f(\Phi_t(x))\,,\quad
        \widetilde S_{h \tau}^{\ell} f(x) := f(\widetilde \Phi_{h}^{\ell}(x))\,,
    \end{equs}
    and their operator norm
    \begin{equs}
        \|S\|_{2 \to 0} := \sup_{\|f\|_{2} = 1} \pa{\sup_{x \in \XX} |S f(x)|}\,,
    \end{equs}
    where $\|f\|_2$ denotes the following norm on $C^2(\mathcal X)$:
    \begin{equ}
      \|f\|_2 := \sup_{x \in \mathcal X, k=0,1,2} \left( \sup_{|\eta| = 1} |D_x^k f(x) \eta|\right)\,.
    \end{equ}
    Under the smoothness hypothesis of this Lemma, by \cite[Thm. 4.5]{Agazzi2023} -- which uses the same notation we just recalled --
    for every $\epsilon' >0$ it holds
    \begin{equs}\label{e:asresult}
        \prob{\limsup_{\ell \to \infty} \|S_t - \widetilde S_{t\tau /\ell}^{\ell}\|_{2\to 0 } > \epsilon' } = 0\,.
    \end{equs}
    \cref{e:bound2} now follows by choosing
    \begin{equ}
        f_{x,t, \epsilon}(\cdot) := 1- \bar \chi_{\epsilon,\Phi_t(x)}(\,\cdot \,),\quad x \in \XX, t > 0,\epsilon >0\,,
    \end{equ}
    with $\bar \chi_{\epsilon,y}\in C^2(\XX)$ having support in the $d_\T$-ball $\mathcal B_{\epsilon/3}(y)$ of radius $\epsilon/3$ around $y\in \XX$, equalling $1$ in $y$ and such that
    \begin{equ}
      \sup_{x \in \XX}\|D_x^2\bar \chi_{\epsilon,y}(x)\| \leq 64 \epsilon^{-2}\quad \text{ and }\quad \sup_{x \in \XX}|D_x\bar \chi_{\epsilon,y}(x)| \leq 16 \epsilon^{-1}\,.
    \end{equ}
    Indeed, setting $\epsilon' = \epsilon^2/128$ we have
    \begin{equs}
  \,&  \prob{\limsup_{m\to \infty}\sup_{x \in \mathcal X}d_\T\pa{\widetilde \Phi_{t/\floor{mt}}^{\floor{mt}}(x) , \Phi_t(x)} > \frac \epsilon 3} \\
  & \qquad\qquad\qquad\quad \qquad\qquad\leq \prob{\limsup_{\ell\to \infty}\sup_{x \in \mathcal X}|f_{x,t,\epsilon}(\widetilde\Phi_{t/\ell}^{\ell}(x))|>\frac{1}2}\\
    & \qquad\qquad\qquad\quad \qquad\qquad= \prob{\limsup_{\ell\to \infty}\sup_{x \in \mathcal X}|(S_t - \widetilde S_{t\tau/\ell}^{\ell})2 \epsilon' f_{x,t,\epsilon}(x)|>\epsilon'}\\
    &\qquad\qquad\qquad\quad \qquad\qquad\leq \prob{\limsup_{\ell\to \infty}\|S_t - \widetilde S_{t\tau/\ell}^{\ell}\|_{2 \to 0}>\epsilon'} = 0\,,
    \end{equs}
    the last step following from \cref{e:asresult}.

    We now turn to proving \cref{e:bound1}. To do so we note that, defining
    \begin{equ}\label{e:M}
      M := \sup_{x \in \TN, i \in \{1, \dots, N\}} (|v_i(x)|, \|D_xv_i(x)\|)\,,
    \end{equ}
       we have that for all $t >0 $
    \begin{equ}\label{eq:trajsep}
      \sup_{r \in (0,t)} d_\T(\Phi^{(i)}_t(x), \Phi^{(i)}_t(y)) \leq
      e^{Mt} d_\T(x,y)\,,
    \end{equ}
    so that we can write, uniformly in $x \in \mathcal X$,
    \begin{equs}
    \,& d_\T \pa{\Phi_t^m(x), \widetilde \Phi_{t/\floor{mt}}^{\floor{mt}}(x)} \\
    &\qquad \qquad \qquad \qquad = d_\T\pa{\widetilde \Phi_{1/m}^{\floor{mt}}(x), \widetilde \Phi_{t/\floor{mt}}^{\floor{mt}}(x)}
    \\&  \qquad \qquad \qquad \qquad  \leq \sum_{j=1}^{N\floor{mt}}   d_\T\pa{\widetilde \Phi_{1/m}^{(N\floor{mt}},j)(\widetilde \Phi_{t/\floor{mt}}^{(j-1,1)}(x)), \widetilde \Phi_{1/m}^{(N\floor{mt}},j+1)(\widetilde \Phi_{t/\floor{mt}}^{(j,1)}(x))}\\
    & \qquad \qquad \qquad \qquad \leq e^{\frac Mm \sum_{j=1}^{N\floor{mt}} \tau_j}  \sum_{j=1}^{N\floor{mt}} \sup_{y}d_\T\pa{\widetilde \Phi_{1/m}^{(j)}(y), \widetilde \Phi_{t/\floor{mt}}^{(j)}(y)}\\
    & \qquad \qquad \qquad \qquad \leq e^{\frac M{mt} \sum_{j=1}^{N\floor{mt}} \tau_j} M \sum_{j=1}^{N\floor{mt}} |1/m-t/\floor{mt}|\tau_j\\
& \qquad \qquad \qquad \qquad  \leq \left|1- \frac{tm}{\floor{tm}}\right| t e^{ \frac {NM}{N\floor{mt}} \sum_{j=1}^{N\floor{mt}} \tau_j}\frac {NM}{N\floor{mt}} \sum_{j=1}^{N\floor{mt}} \tau_j\,.
    \end{equs}
    Combining the strong law of large numbers for $\frac 1 {\ell}\sum_{k = 1}^{\ell} \tau_k$ in $\ell = N\floor{mt} \to \infty$ with $ tm/\floor{tm} \to 1$ as $m \to \infty$ we have that the right hand side converges almost surely to $0$, proving the desired result.

    Finally, to prove \cref{e:bound0}, we define $\floor{t}_m := \max \left \{ \frac j m ~:~\frac j m \leq t\,,\, j \in \mathbb N\right\}$,
    so that
    $m\floor{t}_m \in \mathbb N$ and by \cref{e:psiandphi} we have $\Phi_t^m(x) = \Phi_{\floor{t}_m}^m(x) = \Psi_{\floor{t}_m}^m(x)$. Then, recalling \cref{e:M} we write
    \begin{equs}
      \sup_{x \in \mathcal X}d_\T\pa{\Phi_t^m(x), \Psi_t^m(x)} = \sup_{x \in \mathcal X}d_\T\pa{\Psi_{\floor{t}_m}^m(x), \Psi_t^m(x)} \leq \frac{\tau_{m \floor{t}_m} N M}{m}
    \end{equs}
    which converges almost surely to $0$ as $m \to \infty$, establishing the claim.
    \end{proof}

    \begin{proof}[Proof of \cref{p:convKsmooth}]
      Defining for every $x \in \mathcal X$
      \begin{equ}
        A_m(\epsilon) := \left\{\sup_{t \in [0,1]} d_\T(\isf_t(x), \Phi_t(x)) \leq \epsilon\right\}\,,
      \end{equ}
      we aim to show that for all $\epsilon > 0$,
      $
        \prob{\limsup_{m \to \infty}A_m(\epsilon)^c} =0\,.
      $

      For a stepsize $\sem = \epsilon^2$ and a tolerance $\Delta(\epsilon) := e^{-NM} \epsilon^3/20$ for $M$ in \cref{e:M}, we introduce the sets
      \begin{equs}
          B_{j,m}(\epsilon) & := \left\{ \sup_{t \in (0, \sem )}  d_\T(\isf_t(\isf_{j \sem }(x)), \isf_{j \sem }(x)) \leq \frac \epsilon 3\right\}\,, \\
        B_{j,m}'(\epsilon) &:=  \left\{  d_\T(\Phi_{\sem }(\isf_{j \sem }(x)), \isf_{\sem }(\isf_{j \sem }(x))) \leq \Delta(\epsilon)\right\}\,.
      \end{equs}
       It is readily checked that, for sufficiently small $\epsilon > 0$, one has
      \begin{equ}\label{e:inclusion}
        \bigcap_{j=0}^{\floor{\sem ^{-1}}}B_{j,m}(\epsilon) \cap B_{j,m}'(\epsilon) \subset A_m(\epsilon)\qquad \text{for all } m \in \N\,.
      \end{equ}
       Indeed, adapting the estimate \cref{eq:trajsep} to trajectories of $\Phi$, and since by triangle inequality for all $k \in \{1, \dots, \floor{\sem ^{-1}}\}$
      \begin{equ}
        d_\T(\Phi_{k\sem }(x), \isf_{k\sem}(x)) \leq \sum_{j = 1}^{k} d_\T(\Phi_{(k-j) {\sem }}(\isf_{j\sem}(x)), \Phi_{(k-(j-1))\sem}(\isf_{(j-1)\sem}(x)))\,,
      \end{equ}
      on $\bigcap_{j=1}^{k}B_{j,m}'(\epsilon)$ for all $k \in \{1, \dots, \floor{\sem ^{-1}}\}$ we can write
      \begin{equ}
        d_\T( \Phi_{k \sem}(x),\isf_{k\sem }(x)) \leq e^{NM} \sum_{j = 1}^{k} d_\T(\Phi_{ \sem }(\isf_{j\sem}(x)), \isf_{  \sem }(\isf_{j\sem}(x))) \leq \frac \epsilon {10}\,.
      \end{equ}
      Combining the above with the definition of $B_{j,m}(\epsilon)$ and the fact that for $\epsilon$ small enough $\sup_{x \in \mathcal X, t \in (0, \sem )}d_\T(x, \Phi_t(x)) \leq NM s < \epsilon/3$  yields \cref{e:inclusion}.

      To conclude, it remains to estimate the probabilities of $B_{j,m}(\epsilon), B_{j,m}'(\epsilon)$: for every $\epsilon > 0$, we have by \cref{e:inclusion} that
      \begin{equs}
        \prob{\limsup_{m \to \infty} A_m(\epsilon)^c}& \leq \prob{\limsup_{m \to \infty} \left(\bigcup_{j=0}^{\floor {\sem^{-1}}} B_{j,m}(\epsilon)^c \cup B_{j,m}'(\epsilon)^c \right)}\\
        & \leq \sum_{j=0}^{\floor {\sem^{-1}}} \prob{  \limsup_{m \to \infty} B_{j,m}(\epsilon)^c }+ \sum_{j=0}^{\floor {\sem^{-1}}} \prob{  \limsup_{m \to \infty}B_{j,m}'(\epsilon)^c }\,,
      \end{equs}
        where the second inequality is a union bound. We finally obtain the desired claim by noting that the second term on the right hand side vanishes by application of \cref{l:marginals}, and that by the strong law of large numbers, recalling the definition of $\sem$ and that $\tau_k\overset{\mathrm{iid}}{\sim} \rho$ with $\mathbb E(\tau_k = 1)$, for the first term we have
      \begin{equs}
\prob{\limsup_{m \to \infty} \sup_{x \in \TN, t \in (0, \sem )}  d_\T(\isf_t(x) - x) > \frac \epsilon 3} & \leq \prob{\limsup_{m \to \infty}  \frac {1 } m \sum_{k = 1}^{sm} M \tau_k  > \frac \epsilon 3 }\\
& = \prob{\limsup_{m \to \infty}  \frac {1} {sm} \sum_{k = 1}^{sm} \tau_k  > \frac 1{3M\epsilon } }=0\,,
\end{equs}
upon choosing $\epsilon < 1/4M$.
    \end{proof}

\subsection{Convergence to the deterministic vortex flow}

In what follows, $G~:~\T \setminus (0,0)$ denotes the (zero-average) Green function of the Laplace operator $-\Delta$ on $\T$ and $K(x)=-\nabla^\perp G(x)$ returns to be the singular interaction of the PV system. We denote by $dx^N$ the Lebesgue (equivalently, Haar) measure on $\TN$.

\begin{theorem}\label{thm:main}
    $dx^N\otimes \PP$-almost surely, for all $t\in [0,1]$ we have
    \begin{equ}
       \isf_t(x) \to \Phi_t(x), \quad \text{as } m\to\infty\,.
    \end{equ}
\end{theorem}

The proof essentially relies on the following bound on vortex distances,
which reprises the classical argument of D\"urr-Pulvirenti \cite{Durr1982}.

\begin{lemma}\label{lem:probbounds}
    There exists a constant $C=C(N)>0$ such that for all $\delta>0$
    \begin{equ}\label{eq:pulvirentata}
        \int_{\TN} dx^N\PP\pa{\min_{m\geq 0}\inf_{t\in [0,1]} \min_{i\neq j} d_\T(\isf_t(x)_i,\isf_t(x)_j) <\delta}\leq C (-\log \delta)^{-1}\,.
    \end{equ}
\end{lemma}

\begin{proof}
    To lighten notation, in the following $C$ denotes a positive $N$-dependent constant possibly varying in each occurrence.
    Since $\isf_t$ is the result of subsequent compositions of the flows $\Phi^{(i)}$,
    the proof reduces to establish the thesis for the latter.
    The function
    \begin{equ}
        L:\TN\setminus \triangle \to [0,\infty)\,,
        \quad
        L(x)=\sum_{i\neq j} G(x_i-x_j)+c\,,
    \end{equ}
    (where $c=c(N)>0$ is a constant to be chosen so that $L\geq 0$, thanks to the fact that $G$ is bounded from below, and we define throughout $\triangle := \{x \in \TN~:~x_i = x_j\text{ for } i \neq j\}$)
    allows to control the minimum distance between vortices, since
    \begin{equ}\label{eq:Lestimate}
        L(x)\leq -C\log\pa{ \min_{i\neq j} d_\T(x_i,x_j)}, \quad x\in \TN\setminus \triangle\,.
    \end{equ}
    Notice that $L\in L^1(\TN)$ as $G\in L^1(\T)$.
    It holds:
    \begin{equ}
        \frac{d}{dt} [L\circ \Phi^{(1)}_t](x)
        =\sum_{i\neq 1} \nabla G(x_i-x_1) \sum_{j\neq 1} \nabla^\perp G(x_j-x_1)\,,
    \end{equ}
    in which the sum on the right-hand side also contains no contribution from the product with $i=j$ due to orthogonality.
    Integrating in time we can thus write, for $t^\ast>0 $,
    \begin{multline*}
        \int_{\TN} \sup_{t\in [0,t^\ast]} (L\circ \Phi^{(1)}_t(x)) dx^N
        =\int_{\TN} L(x) dx^N\\
        + \int_{\TN} \sup_{t\in [0,t^\ast]} \int_0^t \sum_{i\neq j\neq 1} \nabla G(\isf_s(x)_i-\isf_s(x)_1)\nabla^\perp G(\isf_s(x)_j-\isf_s(x)_1) ds dx^N\,.
    \end{multline*}
    In the latter expression, since $\Phi^{(1)}$ preserves $dx^N$,
    we can swap the integral over $\TN$ and the supremum over time;
    we can then use the estimate $|\nabla G (y)|\leq C |y|^{-1}$, $y\in \T$, to control factors of the integrand,
    finally arriving to
    \begin{equ}
        \int_{\TN} \sup_{t\in [0,t^\ast]} (L\circ \Phi^{(1)}_t(x)) dx^N
        \leq C t^\ast\,.
    \end{equ}
    The same argument clearly holds for all $\Phi^{(i)}$, and for compositions of those flows on subsequent time intervals
    as the ones in the definition of $\isf_t$, leading in particular to
    \begin{equ}
        \int_{\TN} \sup_{t\in [0,1]} (L\circ \isf_t(x)) dx^N
        \leq C {\frac 1 {m}}\sum_{i=1}^{mN} \tau_i\,.
    \end{equ}
    \noindent
    With this estimate at hand, the thesis follows from \cref{eq:Lestimate} and Markov inequality.
\end{proof}

\begin{proof}[Proof of \cref{thm:main}]
If $(\Omega,\PP)$ is the probability space on which the random times $\tau_i$ are defined, the measurable subset
\begin{equ}
    A_\delta\coloneqq
    \set{\min_{m\geq 0}\inf_{t\in [0,1]} \min_{i\neq j} d_\T(\isf_t(x)_i,\isf_t(x)_j) <\delta}
    \subset \Omega\times \XX\,,
\end{equ}
is such that on $A_\delta$ the random flow $\isf(x)_i$ does not change if the interaction kernel $K$ is replaced by $K_\delta$ as in \Cref{ssec:smooth}. In particular, conditionally to $A_\delta$, \cref{p:convKsmooth} applies yielding: $dx^N\otimes \PP$-almost surely on $A_\delta$, for all $t\in [0,1]$, $\isf_t(x) \to \Phi_t(x)$ as $ m\to\infty$. The proof is then completed by observing that $\bigcup_{\delta>0} A_\delta  = \XX\times \Omega$, therefore the subset of $\TN\otimes \Omega$ on which the thesis does not hold must be negligible by \cref{eq:pulvirentata}.
\end{proof}

\section{Equilibrium Statistical Mechanics}\label{sec:statmech}

We have exhibited a random dynamical system whose flow $\Psi^m$ converges to that of PV dynamics in the deterministic limit $m\to \infty$.
We conclude the present note with some remarks on the compatibility of the flow $\Psi^m$ with the statistical mechanics of PVs. We refer to \cite{Robert1991} for a survey on classical statistical mechanics approach to 2D turbulent phenomena, to \cite{Bouchet2012} for a more recent account, and  to \cite{Eyink1993} for the  relevance to microcanonical ensembles of PVs.

The interaction energy of the PV system,
\begin{equ}
    H(x_1,\dots,x_N)=\sum_{i\neq j} \xi_i \xi_j G(x_i-x_j),
\end{equ}
corresponds to the (renormalized) kinetic energy of the fluid, and it acts as the Hamiltonian function of \cref{eq:pv} regarded as the Hamilton equations in conjugate coordinates $(x_{j,1},\xi_j x_{j,2})$. Combined with the fact that the PV flow $\Phi$ (out of the negligible set of singular initial configurations) is the flow of a divergence-less vector field, and as such preserves $dx^N$ by Liouville theorem, this allowed Onsager \cite{Onsager1949} to consider canonical and microcanonical ensembles preserved by $\Phi$. Specifically,
\begin{equ}\label{eq:canonical}
    \nu_\beta(dx^N)=\frac1{Z_{\beta,N}} e^{-\beta H(x_1,\dots,x_N)},\quad
    Z_{\beta,N}=\int_{\TN} e^{-\beta H(x_1,\dots,x_N)} dx^N,
\end{equ}
is well defined (\emph{i.e.} $Z_{\beta,N}<\infty$) for inverse temperature $\beta < \frac{4\pi}{\min_i |\xi_i|}$
(\emph{cf.} \cite[Section 2]{RomitoCMP}) and defines an invariant measure of \cref{eq:pv}.
On the other hand, conditioning $dx^N$ to an energy level set $\{H=E\}$ one can introduce the microcanonical ensemble
\begin{equ}
    \mu_E(dx^N) =\frac{1}{Z_{E,N}} \delta \pa{H(x_1,\dots,x_N)-E} dx^N,
\end{equ}
$Z_{E,N}$ being the Lebesgue measure of $\{H(x_1,\dots,x_N)=E\}\subset \XX$.
For high enough energy $E\gg1$, Onsager predicted that, under the microcanonical ensemble, typical configurations of vortices behave similarly to samples from a negative-temperature Canonical ensemble, \emph{i.e.} $\beta<0$ in \cref{eq:canonical}. Under the latter distribution, for large enough $|\beta|$ so to prevent statistical Lebesgue repulsion (\emph{cf.} \cite[Section 2.3]{Bodineau1999}), typical configurations should exhibit aggregation of same-sign vortices, as proximity of vortices with different signs is penalized by the density $e^{-\beta H}$.
This should allow the use of PV statistical ensembles to describe the formation of coherent structures in 2D turbulent flows. At present, this remains mostly conjectural as far as rigorous results are concerned, and we shall rather refer to numerical studies such as \cite{Dritschel2015,Kanai2021} for a contemporary viewpoint.

The single component flow $\Phi^{(i)}$ is in fact the flow of the vector field $\nabla^\perp_i H$, thus $H$ is a first integral of motion for all $\Phi^{(i)}$'s, and consequently for the random flows $\Phi^m$, $\Psi^m$. Since $\Phi^{(i)}$ is still the flow of a divergence-less vector field, Liouville theorem applies and the measure invariance arguments just outlined can be repeated for $\Phi^m$, $\Psi^m$. As a consequence, the latter are completely equivalent to \cref{eq:pv} from the point of view of equilibrium states, while being simpler as far as the time evolution is concerned.
Let us also stress again the fact that these random flows are well-defined for \emph{all} initial PV configurations, therefore singular dynamics are completely ruled out in this setting. Incidentally, we observe that this possibly introduces a new tool in the study of the continuation of PV dynamics after collapse via stochastic regularization (\emph{cf.} \cite{Grotto2023}).

Further insight on the stability of vortex interactions is necessary in order to fully replicate the results of \cite{Agazzi2023} for split PV flows, but the splitting approach reduces the problem to the analysis of the evolution of a single PV in a fixed configuration of vortices, thus moving a step forward towards a better understanding of PV dynamics.

 \subsection*{Acknowledgements}
 AA acknowledges partial support by the Italian Ministry for University and Research through PRIN grant \emph{ConStRAINeD}, and by the University of Pisa, through project PRA $2022\_85$. FG was supported by the project \emph{Mathematical methods for climate science} funded by PON R\&I 2014-2020 (FSE REACT-EU). JCM thanks the NSF RTG grant DMS-2038056 for general support and the Visiting Fellow program at the Mathematics Department, University of Pisa.

\bibliography{biblio.bib}{}
\bibliographystyle{plain}

\end{document}